\documentclass[10pt,a4paper,reqno]{amsart}
\usepackage{geometry}
\usepackage{amsmath}
\usepackage{amssymb}
\usepackage{amsthm}
\usepackage{xcolor}
\usepackage{verbatim}
\usepackage{tabularx}
\usepackage{mathrsfs}
\usepackage{hyperref}

\theoremstyle{theorem}
\newtheorem{theorem}{Theorem}[section]
\newtheorem{lemma}[theorem]{Lemma}
\newtheorem{proposition}[theorem]{Proposition}
\newtheorem{corollary}[theorem]{Corollary}

\theoremstyle{remark}
\newtheorem{remark}[theorem]{Remark}
\numberwithin{equation}{section}

\theoremstyle{definition}

\newcommand{\R}{\mathbb{R}}

\newcommand{\Riem}{\mathrm{Riem}}
\newcommand{\Ricc}{\mathrm{Ric}}

\newcommand{\di}{\mathrm{d}}

\makeatletter
\newcommand*\owedge{\mathpalette\@owedge\relax}
\newcommand*\@owedge[1]{
	\mathbin{
		\ooalign{
			$#1\m@th\bigcirc$\cr
			\hidewidth$#1\m@th\wedge$\hidewidth\cr
		}
	}
}
\makeatother

\title{A sharp Eells-Sampson type theorem under positive sectional curvature upper bounds}

\author{Giulio Colombo}
\address{Dipartimento di Matematica e Applicazioni ``R. Caccioppoli'', Universit\`a degli Studi di Napoli Fe\-de\-ri\-co II, Via Vicinale Cupa Cintia 21, I-80126 Napoli, Italy}
\email{giulio.colombo@unina.it}

\author{Marco Mariani}
\address{Dipartimento di Matematica ``F. Enriques", Universit\`a degli Studi di Milano, Via Saldini 50, I-20133 Milano, Italy}
\email{marco.mariani1@unimi.it}

\author{Marco Rigoli}
\address{Dipartimento di Matematica ``F. Enriques", Universit\`a degli Studi di Milano, Via Saldini 50, I-20133 Milano, Italy}
\email{marco.rigoli@unimi.it}

\thanks{}

\date{\today}

\begin{document}

\maketitle

\begin{abstract}
	We prove an extension of Eells and Sampson's rigidity theorem for harmonic maps from a closed manifold of non-negative Ricci curvature to a manifold of non-positive sectional curvature. We give an application of our result in the setting of harmonic-Einstein (or Ricci-harmonic) metrics and as a consequence we recover a classical rigidity result of Hamilton for the problem of prescribed positive definite Ricci curvature.
\end{abstract}

\bigskip



\section{Introduction}

It is well known from the seminal work of J. Eells and J. H. Sampson, \cite{es64}, that any harmonic map $\varphi : (M,g) \to (N,h)$ between a closed Riemannian manifold $(M,g)$ of non-negative Ricci curvature and a Riemannian manifold $(N,h)$ of non-positive sectional curvature is a totally geodesic map, that is, it carries geodesics of $M$ to geodesics of $N$. Furthermore, if $\varphi$ is non-constant then $\Ricc_g$ cannot be positive definite at any point on $M$ and, depending on the (constant) value of the rank of $\di\varphi : TM \to TN$, we have either
\begin{itemize}
	\item [i)] $\mathrm{rank}(\di\varphi) = 1$ and $\varphi$ maps $M$ onto a closed geodesic of $N$, or
	\item [ii)] $\mathrm{rank}(\di\varphi) \geq 2$ and the sectional curvature of $(N,h)$ vanishes on all $2$-planes contained in the subbundle $\di\varphi(TM) \subseteq TN$, so that $\varphi$ maps $M$ onto a closed, flat, totally geodesic submanifold of $N$.
\end{itemize}
In both cases i)-ii), if $\mathrm{rank}(\di\varphi) = \dim M$, that is, if $\varphi$ is a totally geodesic immersion, then $M$ itself must also be flat and if $M$ is irreducible then one further concludes that the immersion $\varphi : M \to N$ is homothetic, \cite[Corollary 2.4]{vi70}, that is,
\[
	\varphi^\ast h = \mu g
\] 
for some constant $\mu>0$, while this is not generally true if $M$ is reducible (for an example, consider the case of two flat tori $(M,g) = \R^k/\Lambda_1$ and $(N,h) = \R^k/\Lambda_2$ of equal dimension $k\geq 2$ defined by non-homothetic lattices $\Lambda_1,\Lambda_2\subseteq\R^k$, and $\varphi : M \to N$ the totally geodesic, affine, necessarily non-homothetic diffeomorphism induced by any affine map $\psi : \R^k \to \R^k$ satisfying $\psi(\Lambda_1) = \Lambda_2$). On the other hand, if $1\leq\mathrm{rank}(\di\varphi) < \dim M$, then as a consequence of a general structure theorem for totally geodesic maps between complete manifolds due to J.~Vilms, \cite[Theorem 2.2]{vi70}, $\varphi$ factors as the composition $\varphi = \varphi_0 \circ \pi$ of a totally geodesic submersion $\pi : M \to B$ onto a closed, flat manifold $B$ of dimension $r = \mathrm{rank}(\di\varphi)$ followed by a totally geodesic immersion $\varphi_0 : B \to N$. Moreover, if this latter case is verified then $M$ is necessarily reducible, \cite[Proposition 2.3]{vi70}, and if $M$ is also simply connected then $M$ is in fact a Riemannian product and $\pi : M \to B$ is just the projection onto one of the factors, \cite[Corollary 3.7]{vi70}.

The above theorem, together with the subsequent observations, can be proved by analysis of the Bochner identity for harmonic maps (see formula 3.13 in \cite{el78})
\begin{equation} \label{boch}
	\frac{1}{2} \Delta|\di\varphi|^2 = |\nabla\di\varphi|^2 + Q(\di\varphi)
\end{equation}
where $|\di\varphi|^2 = g^{ij} h_{ab} \varphi^a_i \varphi^b_j$ is twice the energy density of $\varphi$, $\nabla\di\varphi : TM \otimes TM \to TN$ is the second fundamental form of the map $\varphi$, defined as the covariant derivative of $\di\varphi$ regarded as a section of $T^\ast M \otimes \varphi^{-1}TN$ equipped with the connection $\nabla^g \otimes \varphi^\ast\nabla^h$ induced by the Levi-Civita connections $\nabla^g$ and $\nabla^h$ of $M$ and $N$, and $Q(\di\varphi)$ is a Ricci curvature term given by
\[
	Q(\di\varphi) = g^{ik} g^{jl} h_{ab} R_{ij} \varphi^a_k \varphi^b_l - g^{ik} g^{jl} R^N_{abcd} \varphi^a_i \varphi^b_j \varphi^c_k \varphi^d_l \, .
\]
Here $g_{ij}$, $R_{ij}$, $h_{ab}$, $R^N_{abcd}$ and $\varphi^a_i = \partial(y^a\circ\varphi)/\partial x^i$ denote the local components of $g$, $\Ricc_g$, $h$, $\Riem_h$ and $\di\varphi$ in local charts $\{x^i\}$ and $\{y^a\}$ for $M$ and $N$, respectively, and $g^{ij}$ stand as usual for the coefficients of the inverse matrix of $(g_{ij})$. We recall that the condition of $\varphi$ being totally geodesic is equivalent to $\nabla\di\varphi=0$, since for any curve $\gamma : (a,b) \to M$ the accelerations $\ddot\gamma = \nabla^g_{\dot\gamma} \dot\gamma$ and $\ddot\sigma = \nabla^h_{\dot\sigma} \dot\sigma$ of $\gamma$ and $\sigma := \varphi\circ\gamma$ are related by
\[
	\ddot\sigma = \di\varphi(\ddot\gamma) + (\nabla\di\varphi)(\dot\gamma,\dot\gamma) \, .
\]
Under assumptions $\Ricc_g \geq 0$ and $\sec_h \leq 0$ the term $Q(\di\varphi)$ is non-negative, so $|\di\varphi|^2$ is a subharmonic function on the closed manifold $M$ and then it must be a constant function by the maximum principle. Thus, the right hand side of \eqref{boch} vanishes and one deduces $\nabla\di\varphi = 0$ and $Q(\di\varphi) = 0$, so in particular $\varphi$ must be totally geodesic; in case $\varphi$ is non-constant, further conclusions on the geometry of $\varphi$ follow by refined analysis of the implications of the equation $Q(\di\varphi) = 0$ (for completeness, we include a detailed argument at the end of the paper, see Proposition \ref{prop_esv}).

\bigskip

In this note we prove a result related to Eells and Sampson's theorem, replacing the condition $\Ricc_g \geq 0$ and $\sec_h \leq 0$ with the assumption that
\[
	\Ricc_g \geq (m-1) K \, \varphi^\ast h \qquad \text{and} \qquad \sec_h \leq K
\]
for some constant $K>0$, where $m = \dim M$. In case $K=0$ this would clearly reduce to Eells and Sampson's condition. In our setting, we still conclude that $\varphi$ is totally geodesic if harmonic, with a more stringent conclusion in the case of a non-constant map.

\begin{theorem} \label{thm_hE_N}
	Let $\varphi : (M^{m\geq 2},g) \to (N,h)$ be a harmonic map between Riemannian manifolds. Assume that $M$ is closed and that there exists $K>0$ such that
	\begin{equation} \label{Ric_ineq}
		\Ricc_g \geq (m-1) K \, \varphi^\ast h \qquad \text{and} \qquad \sec_h \leq K \, .
	\end{equation}
	Then $\varphi$ is a totally geodesic map. In particular, either
	\begin{itemize}
		\item [i)] $\varphi$ is constant, or
		\item [ii)] $\varphi$ is a homothetic immersion, $g$ has positive
		constant curvature and the inequalities in \eqref{Ric_ineq} hold with equality sign on $M$ and $\varphi(M)$, respectively, that is,
		\[
			\Ricc_g = (m-1) K \, \varphi^\ast h \qquad \text{and} \qquad \sec_h(\Pi) = K
		\]
		for any $2$-plane $\Pi$ contained in $\di\varphi(TM)\subseteq TN$.
	\end{itemize}
\end{theorem}

\begin{remark}
	In case ii) of Theorem \ref{thm_hE_N} the constant value $K_g>0$ of the sectional curvatures of $g$ and the homothety parameter $\mu>0$ such that $\varphi^\ast h = \mu g$ are related by
	\[
		K_g = \mu K
	\]
	and the image $\varphi(M)\subseteq N$ is a totally geodesic submanifold with constant sectional curvature $K$.
\end{remark}

\begin{remark}
	We remark that the validity of \eqref{Ric_ineq} for some $K>0$ is invariant with respect to rescalings of $g$ and $h$ (as is harmonicity of $\varphi$): if $(M,g)$, $(N,h)$ and $\varphi : M \to N$ satisfy \eqref{Ric_ineq} for some $K>0$, then for any two constants $c_1,c_2>0$ the metrics $\bar g = c_1 g$ and $\bar h = c_2 h$ satisfy $\Ricc_{\bar g} = \Ricc_g$ and $\sec_{\bar h} = c_2^{-1} \sec_h$, so for $\bar K = c_2^{-1} K > 0$ we have
	\[
		\Ricc_{\bar g} \geq (m-1) \bar K \, \varphi^\ast \bar h \qquad \text{and} \qquad \sec_{\bar h} \leq \bar K \, .
	\]
\end{remark}

Our main motivation for considering an inequality on $\Ricc_g$ as in \eqref{Ric_ineq} comes from the study of \emph{harmonic-Einstein} (or \emph{Ricci-harmonic}) metrics: given a fixed Riemannian manifold $(N,h)$, a metric $g$ on a smooth manifold $M$ is said to be harmonic-Einstein (with respect to $h$) if there exists a harmonic map $\varphi : (M,g) \to (N,h)$ such that
\begin{equation} \label{Ric_gh}
	\Ricc_g - \alpha\,\varphi^\ast h = \lambda \, g
\end{equation}
for some constants $\alpha\in\R\setminus\{0\}$ and $\lambda\in\R$. If $g$ is harmonic-Einstein and $\varphi : (M,g) \to (N,h)$ is a harmonic map realizing the defining condition \eqref{Ric_gh} for $g$, then we also refer collectively to the pair $(g,\varphi)$ as a harmonic-Einstein structure on $M$.

Harmonic-Einstein structures arise as fixed points of the (normalized) harmonic-Ricci flow studied by R. Buzano, \cite{mull12}, a geometric flow obtained by coupling Hamilton's Ricci flow with Eells-Sampson's heat flow for harmonic maps. In the last decade it has been proved that many key concepts in the theory of the Ricci flow can be extended with due modifications to the harmonic-Ricci flow, see for instance \cite{chzhu13}, \cite{li14}, \cite{will15}, and that solitons of the harmonic-Ricci flow and \emph{quasi-harmonic-Einstein} (or \emph{quasi-Ricci-harmonic}) metrics (a notion which relates to harmonic-Einstein structures in pretty much the same way as quasi-Einstein metrics relate to Einstein metrics, see the paper \cite{wang16} for a precise definition) exhibit similar features as their Ricci flow and quasi-Einstein counterparts, \cite{will15}, \cite{wang16}, \cite{tad16}, \cite{zhu17}, \cite{zeng18}, \cite{acr21}, \cite{ans21}, \cite{ans22}, \cite{cmr22}. We remark that equation \eqref{Ric_gh} also naturally arises in General Relativity for a \emph{Lorentzian} metric $g$ as the system of Einstein field equations (with possibly non-zero cosmological constant) for a spacetime $(M,g)$ having a wave map $\varphi : (M,g) \to (N,h)$ as gravitational energy source, \cite[Section III.6.5]{cb09}.

In the setting of (Riemannian) harmonic-Einstein structures, Theorem \ref{thm_hE_N} immediately implies the following result.

\begin{theorem} \label{thm_hE}
	Let $(N,h)$ be a Riemannian manifold and let $(g,\varphi)$ be a harmonic-Einstein structure (with respect to $h$) on a closed manifold $M$ of dimension $m\geq 2$. If $\alpha>0$, $\lambda\geq 0$ and
	\begin{equation} \label{KNalpha}
		\sec_h \leq \frac{\alpha}{m-1}
	\end{equation}
	then $\varphi$ is either constant or a homothetic immersion. Moreover, in the latter case necessarily $\lambda=0$, $g$ has constant curvature and \eqref{KNalpha} holds with equality sign on any $2$-plane in $\di\varphi(TM)$.
\end{theorem}

It is worth pointing out that the condition of harmonicity of $\varphi$ is not independent from the validity of \eqref{Ric_gh}: if $(M,g)$ and $(N,h)$ are (semi-)Riemannian manifolds and $\varphi : M \to N$ is any smooth map satisfying \eqref{Ric_gh} for some constants $\alpha\neq 0$ and $\lambda$, then $\varphi$ is conservative, that is,
\[
	\langle\tau(\varphi),\di\varphi\rangle_h = 0 \, ,
\]
see \cite[Proposition 2.15]{acr21}. If $\varphi$ is additionally a submersion, this implies $\tau(\varphi) = 0$. This observation relates the study of harmonic-Einstein metrics with $\lambda=0$ to the problem of the prescribed Ricci curvature. The problem can be stated as follows: given a symmetric $2$-covariant tensor field $h$ on a smooth manifold $M$, to find a Riemannian metric $g$ on $M$ and a constant $c>0$ such that
\begin{equation} \label{R_ch}
	\Ricc_g = c \, h \, .
\end{equation}
(The presence of the parameter $c$ compensates for the lack of homogeneity inherent to the problem.) Local solvability of the problem and regularity of the solutions were first addressed by D. DeTurck, \cite{dt81}, and by DeTurck and J. Kazdan, \cite{dtka81}, respectively. Later, R. Hamilton, \cite{ham84}, and DeTurck and N. Koiso, \cite{dtko84}, established uniqueness and non-existence results for global solutions $g$ of \eqref{R_ch} in particular cases where $h$ is everywhere positive definite (we also refer to \cite[Chapter 5]{bes08} as a further reference on the topic.) In this latter case, $h$ can be regarded as a Riemannian metric itself on $M$, so that equation \eqref{R_ch} takes the form \eqref{Ric_gh} with $\alpha = c$, $\lambda = 0$, $N = M$ and $\varphi = \mathrm{id}_M$, and if $g$ satisfies \eqref{R_ch} then $\mathrm{id}_M : (M,g) \to (M,h)$ is automatically harmonic by the previous observations (see also \cite[Corollary 3.3]{ham84} for a direct proof). In this setting, a direct consequence of Theorem \ref{thm_hE} is the following corollary, which recovers Theorems 4.1 and 4.3 of \cite{ham84}.

\begin{corollary}
	Let $(M,h)$ be a closed Riemannian manifold with sectional curvature $\leq 1$. If there exists a Riemannian metric $g$ on $M$ such that
	\begin{equation} \label{Rgh}
		\Ricc_g = (m-1) h
	\end{equation}
	then the metrics $g$ and $h$ are homothetic (that is, there exists a constant $\mu>0$ such that $g = \mu h$) and the original metric $h$ must have constant sectional curvature exactly $1$ everywhere on $M$.
	
	In particular, if $\sec_h < 1$ somewhere on $M$ then there exists no metric $g$ satisfying \eqref{Rgh}.
\end{corollary}

\section{Proof of Theorem \ref{thm_hE_N}}

To prove Theorem \ref{thm_hE_N} we have to rewrite the term $Q(\di\varphi)$ in the Bochner identity \eqref{boch} as
\begin{equation} \label{Q01_def}
	Q(\di\varphi) = Q_0(\di\varphi) + Q_1(\di\varphi)
\end{equation}
with
\begin{align}
	\label{Q0_def}
	Q_0(\di\varphi) & = g^{ik} g^{jl} h_{ab} R_{ij} \varphi^a_k \varphi^b_l - (m-1)K g^{ik} g^{jl} h_{ad} h_{bc} \varphi^a_i \varphi^b_j \varphi^c_k \varphi^d_l \, , \\
	\label{Q1_def}
	Q_1(\di\varphi) & = (m-1) K g^{ik} g^{jl} h_{ad} h_{bc} \varphi^a_i \varphi^b_j \varphi^c_k \varphi^d_l - g^{ik} g^{jl} R^N_{abcd} \varphi^a_i \varphi^b_j \varphi^c_k \varphi^d_l \, .
\end{align}
The next two lemmas show that $Q_0(\di\varphi)\geq0$ and $Q_1(\di\varphi)\geq0$ provided $\Ricc_g \geq (m-1) K \varphi^\ast h$ and $\sec_h \leq K$, respectively, and also characterizes the equality case $Q_1(\di\varphi) = 0$ when $\sec_h \leq K$.

\begin{lemma} \label{lem_Q0}
	Let $\varphi : (M,g) \to (N,g_N)$ be a smooth map between Riemannian manifolds such that
	\begin{equation} \label{Rick}
		\Ricc_g \geq (m-1) K \, \varphi^\ast h
	\end{equation}
	for some $K\in\R$. Then $Q_0(\di\varphi) \geq 0$ on $M$.
\end{lemma}

\begin{proof}
	The bilinear form $A = \Ricc_g - (m-1) K \, \varphi^\ast h$ is positive semidefinite by \eqref{Rick} and we have
	\[
		Q_0(\di\varphi) = \langle A, \varphi^\ast h \rangle = h_{ab} A^{ij} \varphi^a_i \varphi^b_j
	\]
	where
	\[
		A^{ij} = g^{ik} g^{jl} A_{kl} = g^{ik} g^{jl} (R_{kl} - (m-1) K h_{ab} \varphi^a_k \varphi^b_l) \, .
	\]
	Let $x\in M$ be given and fix local coordinates $\{x^i\}$ on $M$ around $x$ such that the matrix $(A^{ij})$ is diagonal at $x$. Then we have $A^{ij} = 0$ whenever $i\neq j$ and $A^{ii} \geq 0$ for each $1 \leq i \leq m = \dim M$, so
	\[
		Q_0(\di\varphi) = \sum_{i=1}^m A^{ii} h_{ab} \varphi^a_i \varphi^b_i = \sum_{i=1}^m A^{ii} |\di\varphi(\partial_{x^i})|_h^2 \geq 0 \qquad \text{at } \, x \, .
	\]
\end{proof}

\begin{lemma} \label{lem_Q1_struct}
	Let $\varphi : (M^{m\geq 2},g) \to (N,g_N)$ be a smooth map between Riemannian manifolds, with
	\begin{equation} \label{ak}
		\sec_h \leq K
	\end{equation}
	for some $K\geq 0$. Then $Q_1(\di\varphi) \geq 0$ on $M$. Moreover, if $Q_1(\di\varphi) = 0$ at some point $x\in M$, then considering the differential
	\[
		\di\varphi_x : T_x M \to T_{\varphi(x)} N
	\]
	we have either
	\begin{itemize}
		\item [i)] $\mathrm{rank}(\di\varphi_x) \geq 2$ and $\sec_h(\Pi) = K$ for each $2$-plane $\Pi\leq\di\varphi_x(T_x M)$, or
		\item [ii)] $\mathrm{rank}(\di\varphi_x) \leq 1$,
	\end{itemize}
	and if $K>0$ then we further have
	\begin{itemize}
		\item [i')] $\varphi^\ast h$ is a multiple of $g$ at $x$ in case \emph{i)},
		\item [ii')] $\di\varphi_x = 0$ in case \emph{ii)}.
	\end{itemize}
\end{lemma}

\begin{proof}
	Let $x\in M$ be given and $\{e_i\}_{i=1}^m$ be an orthonormal basis for $T_x M$. Let $\{x^i\}$ be normal coordinates centered at $x$ such that $\partial_{x^i} = e_i$ for $1\leq i\leq m$, and let us set $Y_i = \di\varphi(e_i)$ for $1\leq i\leq m$. Then, with respect to any local chart $\{y^a\}$ for $N$ centered at $\varphi(x)$, for each $1\leq i,j\leq m$ we have
	\[
		h_{ad} h_{bc} \varphi^a_i \varphi^b_j \varphi^c_i \varphi^d_j = h(Y_i,Y_j)^2 \, , \qquad R^N_{abcd}\varphi^a_i \varphi^b_j \varphi^c_i \varphi^d_j = \Riem_h(Y_i,Y_j,Y_i,Y_j) \qquad \text{at } \, x
	\]
	(no summation over $i$ or $j$ is intended in the above formulas). For each pair $(i,j)$ we choose a $2$-plane $\Pi_{ij}$ in $T_{\varphi(x)} N$ containing $Y_i$ and $Y_j$ (which is clearly uniquely determined in case $Y_i$ and $Y_j$ are linearly independent) and we let $\kappa_{ij} = \sec_h(\Pi_{ij})$. Then
	\[
		\Riem^N(Y_i,Y_j,Y_i,Y_j) = \kappa_{ij} \left[ h(Y_i,Y_i) h(Y_j,Y_j) - h(Y_i,Y_j)^2 \right] \, .
	\]
	For ease of notation let us set $c_{ij} = h(Y_i,Y_j)$ for each $1\leq i,j\leq m$. Then from the above observations we have
	\begin{equation} \label{Q1k}
		Q_1(\di\varphi) = (m-1) K \sum_{i,j=1}^m c_{ij}^2 - \sum_{i,j=1}^m \kappa_{ij}(c_{ii} c_{jj} - c_{ij}^2) \qquad \text{at } \, x \, .
	\end{equation}
	Noting that
	\[
		\sum_{1\leq i<j\leq m} (c_{ii} - c_{jj})^2 = \frac{1}{2} \sum_{i,j=1}^m (c_{ii} - c_{jj})^2 = (m-1) \sum_{i=1}^m c_{ii}^2 - 2\sum_{1\leq i<j\leq m} c_{ii} c_{jj}
	\]
	we express
	\begin{align*}
		\sum_{1\leq i,j\leq m} c_{ij}^2 & = \sum_{i=1}^m c_{ii}^2 + 2 \sum_{1\leq i<j\leq m} c_{ij}^2 \\
		& = \frac{1}{m-1} \sum_{1\leq i<j\leq m} (c_{ii} - c_{jj})^2 + \frac{2}{m-1} \sum_{1\leq i<j\leq m} c_{ii} c_{jj} + 2 \sum_{1\leq i<j\leq m} c_{ij}^2 \, .
	\end{align*}
	Since we also have $c_{ii} c_{jj} - c_{ij}^2 = 0$ whenever $i=j$, we can restate \eqref{Q1k} as
	\begin{equation} \label{Q1s}
		Q_1(\di\varphi) = \sum_{1 \leq i<j\leq m} \left[ K (c_{ii}-c_{jj})^2 + 2 \left(K - \kappa_{ij}\right)c_{ii}c_{jj} + 2 ((m-1)K + \kappa_{ij}) c_{ij}^2 \right] .
	\end{equation}
	Now the conclusion that $Q_1(\di\varphi)\geq 0$ at $x$ follows since all terms in the sum on the RHS of \eqref{Q1s} are non-negative, that is,
	\begin{equation} \label{Q1ij}
		K (c_{ii}-c_{jj})^2 + 2 \left(K - \kappa_{ij}\right)c_{ii}c_{jj} + 2 ((m-1)K + \kappa_{ij}) c_{ij}^2 \geq 0
	\end{equation}
	for each $1\leq i<j\leq m$. Indeed, if $(i,j)$ is a pair such that $\kappa_{ij}\leq 0$ then we have
	\begin{equation} \label{Q1s1}
		\begin{split}
			K (c_{ii}-c_{jj})^2 & + 2 \left(K - \kappa_{ij}\right)c_{ii}c_{jj} + 2 ((m-1)K + \kappa_{ij})c_{ij}^2 \\
			& = K \left[ (c_{ii}-c_{jj})^2 + 2 c_{ii} c_{jj} + 2(m-1) c_{ij}^2 \right] - 2\kappa_{ij}(c_{ii} c_{jj} - c_{ij}^2) \\
			& = K \left[ c_{ii}^2 + c_{jj}^2 + 2(m-1) c_{ij}^2 \right] - 2\kappa_{ij}(c_{ii} c_{jj} - c_{ij}^2) \\
			& \geq -2\kappa_{ij}(c_{ii} c_{jj} - c_{ij}^2) \geq 0
		\end{split}
	\end{equation}
	since $K\geq0$ and $c_{ii}c_{jj}-c_{ij}^2 \geq 0$ by positivity of $h$, while if $\kappa_{ij} > 0$ then by \eqref{ak} we get
	\begin{equation} \label{Q1s2}
		K (c_{ii}-c_{jj})^2 + 2 \left(K - \kappa_{ij}\right)c_{ii}c_{jj} + 2 ((m-1)K + \kappa_{ij})c_{ij}^2 \geq 2 \left(K - \kappa_{ij}\right)c_{ii}c_{jj} \geq 0 \, .
	\end{equation}
	
	Now, suppose that $Q_1(\di\varphi) = 0$ at $x$. We shall prove that if $\mathrm{rank}(\di\varphi_x) \geq 2$ then $\sec_h(\Pi) = K$ for each $2$-plane $\Pi\leq T_{\varphi(x)} N$, and subsequently we shall also prove that if $K>0$ then one further concludes that i') or ii') hold, according to whether $\mathrm{rank}(\di\varphi_x) \geq 2$ or $\mathrm{rank}(\di\varphi_x)\leq 1$.
		
	First, note that if $Q_1(\di\varphi) = 0$ then, whichever orthonormal basis $\{e_i\}$ for $T_x M$ is chosen, inequality \eqref{Q1ij} must hold with equality sign for each $1\leq i<j\leq m$, so in particular it must be
	\begin{equation} \label{Q112}
		K (c_{11}-c_{22})^2 + 2 \left(K - \kappa_{12}\right)c_{11}c_{22} + 2 ((m-1)K + \kappa_{12}) c_{12}^2 = 0 \, .
	\end{equation}
	If $\mathrm{rank}(\di\varphi_x) \geq 2$ and $\Pi\leq T_{\varphi(x)} N$ is a $2$-plane, then we can suppose that $\{e_i\}$ has been chosen so that $\Pi = \Pi_{12} = \mathrm{span}\{Y_1,Y_2\}$. Then $c_{11}, c_{22}, c_{11}c_{22}-c_{12}^2 > 0$ by linear independence of $Y_1$ and $Y_2$. We distinguish the cases $K=0$ and $K>0$.
	\begin{itemize}
		\item If $K=0$, then necessarily $\kappa_{12}\leq 0$ by assumption \eqref{ak}, and since the last inequality in \eqref{Q1s1} (for $i=1$, $j=2$) must hold with equality sign we conclude that $\kappa_{12} = 0$.
		\item If $K>0$, then $K[c_{11}^2+c_{22}^2+2(m-1)c_{12}^2] > 0$, so if it were $\kappa_{12}\leq 0$ then the second-to-last inequality in \eqref{Q1s1} would be strict and therefore \eqref{Q1ij} would not hold with equality sign for $i=1$, $j=2$. Therefore, it must be $\kappa_{12}>0$ and all inequalities in \eqref{Q1s2} must hold with equality sign. Since $c_{11}c_{22}>0$, this yields $\kappa_{12} = K$.
	\end{itemize}
	In both cases, we obtained that $\sec_h(\Pi) = K$.
	
	Lastly, suppose that $K>0$. We first observe that $r = \mathrm{rank}(\di\varphi_x) \in \{0,\dots,m\}$ must be either $0$ or $m$: by contradiction, if it were $1\leq r\leq m-1$ then we could choose an orthonormal basis $\{e_i\}$ for $T_x M$ such that $Y_1 = \di\varphi(e_1) \neq 0$ and $Y_2 = \di\varphi(e_2) = 0$, yielding $c_{11}>0$, $c_{22} = c_{12} = 0$, and therefore the LHS of \eqref{Q112} would be equal to $K c_{11}^2>0$, contradiction. Now, if $r=0$ then clearly $\di\varphi_x = 0$ and this proves ii'). On the other hand, if $r=m$ then for any orthonormal basis $\{e_i\}$ for $T_x M$ we have that $Y_i = \di\varphi(e_i)$, $Y_j = \di\varphi(e_j)$ are linearly independent, so $\kappa_{ij}$ is the sectional curvature of a $2$-plane in $T_{\varphi(x)} N$ and we have $\kappa_{ij} = K > 0$ from what we proved above. Since \eqref{Q1ij} must hold with equality sign, we conclude that in this case $c_{ii} = c_{jj}$ and $c_{ij} = 0$ for $1\leq i<j\leq m$. In other words, the matrix $(c_{ij})_{1\leq i,j\leq m}$ representing the bilinear form $(\varphi^\ast h)_x$ with respect to the orthonormal basis $\{e_i\}$ is a (positive) multiple of the identity matrix, and so $\varphi^\ast h$ is a positive multiple of $g$ at $x$, which proves i').
\end{proof}

\begin{proof}[Proof of Theorem \ref{thm_hE_N}]
	We rewrite the Bochner identity \eqref{boch} as
	\begin{equation} \label{hE_B1}
		\frac{1}{2} \Delta|\di\varphi|^2 = |\nabla\di\varphi|^2 + Q_0(\di\varphi) + Q_1(\di\varphi)
	\end{equation}
	with $Q_0(\di\varphi)$ and $Q_1(\di\varphi)$ as in \eqref{Q0_def}-\eqref{Q1_def}. By Lemmas \ref{lem_Q0} and \ref{lem_Q1_struct}, under the assumption \eqref{Ric_ineq} we have that the three terms appearing on the RHS of \eqref{hE_B1} are all non-negative. So $|\di\varphi|^2$ is a subharmonic function on the closed manifold $M$. By the maximum principle, $|\di\varphi|^2$ must be constant on $M$. Consequently, the RHS of \eqref{hE_B1} must vanish and therefore
	\begin{equation} \label{hE_B2}
		\nabla\di\varphi = 0 \, , \qquad Q_0(\di\varphi) = 0 \, , \qquad Q_1(\di\varphi) = 0 \qquad \text{on } \, M \, .
	\end{equation}
	In particular, $\varphi$ is a totally geodesic map ($\nabla\di\varphi \equiv 0$). If $\varphi$ is constant, then there is nothing else to show. So, suppose that $\varphi$ is non-constant. Then $|\di\varphi|^2 > 0$. From $\di\varphi\neq0$ and $Q_1(\di\varphi) = 0$ we deduce from Lemma \ref{lem_Q1_struct} that $\di\varphi$ has rank $m = \dim M$ everywhere on $M$ and $\varphi : M \to N$ is a conformal immersion, that is, there exists a positive function $\mu : M \to (0,+\infty)$ such that
	\[
		\varphi^\ast h = \mu g
	\]
	on $M$. Tracing both sides of this equality we get $|\di\varphi|^2 = m \mu$, so $\mu = |\di\varphi|^2/m$ is a positive constant and the immersion $\varphi : M \to N$ is in fact homothetic. Since $\varphi$ is also totally geodesic and $N$ has constant positive sectional curvature $K$ on any $2$-plane tangent to $\varphi(M)$, by Gauss's equations we have that $M$ also has constant positive sectional curvature. Moreover, the constant value of the sectional curvature of $M$ is
	\[
		K_g = \mu K
	\]
	so that the Ricci curvature of $M$ satisfies
	\[
		\Ricc_g = (m-1) K \mu g = (m-1) K \, \varphi^\ast h \, .
	\]
\end{proof}

As anticipated in the introduction, we conclude the paper by giving a proof of the further claims about the geometry of the map $\varphi$ in the non-constancy case of the original theorem by Eells and Sampson.

\begin{proposition}[\cite{es64},\cite{el78},\cite{vi70}] \label{prop_esv}
	Let $\varphi : (M,g) \to (N,h)$ be a harmonic map between Riemannian manifolds. Assume that $M$ is closed and that
	\begin{equation}
		\Ricc_g \geq 0 \qquad \text{and} \qquad \sec_h \leq 0 \, .
	\end{equation}
	Then $\varphi$ is a totally geodesic map and either
	\begin{itemize}
		\item [i)] $\varphi$ is constant, or
		\item [ii)] $\varphi$ is non-constant, $\di\varphi$ has constant rank $r\geq1$ and there exists a closed, flat $r$-dimensional manifold $(B,g_B)$ such that $\varphi$ factors as the composition $\varphi = \varphi_0\circ\pi$ of a totally geodesic Riemannian submersion $\pi : M \to B$ and of a totally geodesic immersion $\varphi_0 : B \to N$. In particular:
		\begin{itemize}
			\item [a)] if $r=1$  then $\varphi(M) = \varphi_0(B)$ is the image of a closed geodesic in $N$;
			\item [b)] if $r\geq2$ then
			\[
				\Ricc_g(X,Y) = 0 \qquad \text{and} \qquad \sec_g(\Pi_0) = \sec_h(\Pi_1) = 0 
			\]
			for any pair of vectors $X,Y$ in the horizontal distribution $\ker(\di\varphi)^\bot \subseteq TM$ and for any pair of $2$-planes $\Pi_0 \leq \ker(\di\varphi)^\bot$ and $\Pi_1 \leq \di\varphi(TM)$.
		\end{itemize}
	\end{itemize}
\end{proposition}

\begin{proof}
	We reason as in the first part of the proof of Theorem \ref{thm_hE_N} to conclude that also in this case we have
	\[
		|\di\varphi|^2 = \text{constant} \, , \qquad \nabla\di\varphi = 0 \, , \qquad Q_0(\di\varphi) = 0 \, , \qquad Q_1(\di\varphi) = 0
	\]
	where in this case $Q_0(\di\varphi)$ and $Q_1(\di\varphi)$ are
	\[
		Q_0(\di\varphi) = g^{ik} g^{jl} h_{ab} R_{ij} \varphi^a_k \varphi^b_l \, , \qquad
		Q_1(\di\varphi) = - g^{ik} g^{jl} R^N_{abcd} \varphi^a_i \varphi^b_j \varphi^c_k \varphi^d_l \, .
	\]
	If $|\di\varphi|^2 = 0$ then $\varphi$ is constant, otherwise we have that $\di\varphi : TM \to TN$ has constant rank due to condition $\nabla\di\varphi = 0$ (see for instance \cite[page 9]{el78}) and setting $r = \mathrm{rank}(\di\varphi) \geq 1$ we have by \cite[Theorem 2.2]{vi70} that there exists a closed $r$-dimensional Riemannian manifold $(B,g_B)$ such that $\varphi$ factors as the composition of a totally geodesic Riemannian submersion $\pi : M \to B$ and of a totally geodesic immersion $\varphi_0 : B \to N$. From condition $Q_0(\di\varphi) = 0$ it follows in this case that $\Ricc_g$ vanishes on the horizontal distribution $\ker(\di\varphi)^\bot \subseteq TM$. If $r=1$ then $\varphi(M)$ is one-dimensional and by the definition of totally geodesic map and compactness of $M$ one concludes that it must be a closed geodesic of $N$; moreover, in this case $B$ is obviously flat. Hence, it remains to show that also in case $r\geq 2$ the induced metric on $B$ is flat and $\sec_g(\Pi_0) = \sec_h(\Pi_1) = 0$ for any $2$-planes $\Pi_0\leq\ker(\di\varphi)^\bot$ and $\Pi_1 \leq \di\varphi(TM)$. The claim about $\Pi_1$ is a direct consequence of Lemma \ref{lem_Q1_struct} and of $Q_1(\di\varphi) = 0$. We now prove that $B$ is flat. (If the totally geodesic immersion $\varphi_0 : B \to N$ is also isometric, this is obviously a consequence of Gauss's equations together with $\Riem_h = 0$ on $\di\varphi_0(TB) = \di\varphi(TM)$.) Let $x\in B$ be given and let $\{x^i\}_{1\leq i\leq r}$ be local normal coordinates on $B$ centered at $x$. By the rank theorem, we can choose a local chart $\{y^a\}_{1\leq a\leq n=\dim N}$ on $N$ around $\varphi_0(x)$ such that
	\begin{equation} \label{phi0ai}
		(\varphi_0)^a_i := \frac{\partial (y^a\circ\varphi_0)}{\partial x^i} = \delta^a_i \qquad \text{at } \, x \qquad \text{for } \, 1 \leq i \leq r, \, 1 \leq a \leq n
	\end{equation}
	with $\delta$ the Kronecker delta, and
	\begin{equation} \label{hab}
		h^{ab} = 0 \qquad \text{at } \, \varphi_0(x) \qquad \text{whenever } \, 1 \leq a \leq r < b \leq n \, .
	\end{equation}
	Then, from the commutation Ricci identities for the components $(\varphi_0)^a_{ijk}$ of the second covariant derivative $\nabla\nabla\di\varphi_0 : TB^{\otimes 3} \to TN$ of $\di\varphi_0$, see for example \cite[formula 1.174]{amr16}, we have
	\[
		(\varphi_0)^a_{ijk} = (\varphi_0)^a_{ikj} + g_B^{tl} R^B_{tijk} (\varphi_0)^a_l - h^{ae} R^N_{ebcd} (\varphi_0)^b_i (\varphi_0)^c_j (\varphi_0)^d_k
	\]
	for any $1\leq i,j,k \leq r$ and $1\leq a\leq n$, with repeated summation intended over $1\leq t,l\leq r$ and $1\leq b,c,d,e\leq n$. From \eqref{phi0ai} and \eqref{hab} and normalcy of $\{x^i\}$ at $x$ we further have
	\[
		(\varphi_0)^a_{ijk} = (\varphi_0)^a_{ikj} + R^B_{aijk} - \sum_{e=1}^r h^{ae} R^N_{eijk} \qquad \text{at } \, x
	\]
	for any $1 \leq a,i,j,k \leq r$. Since $\varphi_0$ is totally geodesic we have $(\varphi_0)^a_{ijk} = (\varphi_0)^a_{ikj} = 0$, and since $N$ has vanishing sectional curvature on any $2$-plane contained in $\di\varphi(TM) \equiv \di\varphi_0(TB)$ we also have $R^N_{eijk} = 0$ for any $1 \leq e,i,j,k \leq r$ by polarization. Therefore, we conclude $R^B_{aijk} = 0$ for each $1\leq a,i,j,k\leq r$, that is, the Riemann tensor of $B$ vanishes. Lastly, by \cite[Theorem 3.3]{vi70}, since $\pi : M \to B$ is a totally geodesic Riemannian submersion we have that the horizontal distribution $\ker(\di\pi)^\bot$ is integrable, hence involutive, so from O'Neill's formula (see \cite[Corollary 1]{on66} or \cite[Corollary 9.29]{bes08}) we have $\sec_g(\Pi_0) = \sec_{g_B}(\di\pi(\Pi_0)) = 0$ for any $2$-plane $\Pi_0\leq\ker(\di\varphi)^\bot$.
\end{proof}

\bigskip

\bibliographystyle{plain}

\end{document}